\documentclass{amsart}

\usepackage{amsmath,amssymb,amsthm, graphicx}
\newtheorem{theorem}{Theorem}[section]
\newtheorem{lemma}[theorem]{Lemma}
\newtheorem{remark}[theorem]{Remark}
\newcommand{\D}{\mathbb{D}}

\newcommand{\R}{\mathbb{R}}
\newcommand{\dist}{\operatorname{dist}}

\usepackage{titlesec}

\titleformat{\section}
  {\normalfont\large\bfseries\centering}
  {\thesection}{1em}{}

\titleformat{\subsection}
  {\normalfont\normalsize\bfseries\centering}
  {\thesubsection}{1em}{}

\titleformat{\subsubsection}
  {\normalfont\normalsize\bfseries\centering}
  {\thesubsubsection}{1em}{}

\usepackage{fancyhdr}
\title{Lipschitz regularity of harmonic quasiconformal maps
between Lyapunov domains in $\mathbb{R}^n$}

\author{Anton Gjokaj}
\address{University of Montenegro, Faculty of Natural Sciences and Mathematics,
Podgorica, Montenegro}
\email{antondj@ucg.ac.me}

\author{David Kalaj}
\address{University of Montenegro, Faculty of Natural Sciences and Mathematics,
Podgorica, Montenegro}
\email{davidk@ucg.ac.me}
\date{}

\begin{document}
\maketitle
\begin{abstract}
We prove that every sense-preserving harmonic $K$--quasiconformal homeomorphism
$f\colon D\to\Omega$ between Lyapunov domains (equivalently, bounded $C^{1,\alpha}$ domains)
in $\mathbb{R}^n$, $\alpha\in(0,1]$, is globally Lipschitz on $\overline D$.
The argument is based on a boundary iteration scheme: an initial H\"older modulus for the
boundary trace (coming from quasiconformality) is improved via the $C^{1,\alpha}$ graph
representation of $\partial\Omega$, yielding higher H\"older regularity for the normal component.
This boundary gain is converted into a near-boundary gradient bound for harmonic functions
through a basepoint boundary H\"older-to-gradient estimate obtained by flattening the boundary and using
local harmonic-measure bounds. Quasiconformality then propagates the resulting control from one
component to the full differential, and iteration gives boundedness of $|Df|$ up to the boundary.
Along the way we briefly survey several standard tools from the theory of quasiconformal harmonic
mappings (QCH), including boundary H\"older continuity, distortion of derivatives, and
boundary-to-interior propagation principles that enter the iteration.
\end{abstract}
\section{Quasiconformal harmonic mappings: an introduction}

A \emph{harmonic mapping} is a function $f:D\subset \R^n\to \R^n$ whose components are harmonic
(equivalently, $\Delta f=0$ in the distributional sense).

A \emph{$K$--quasiconformal mapping} is a
sense-preserving homeomorphism $f\in W^{1,n}_{\mathrm{loc}}(D,\R^n)$ such that there exists
$K\ge 1$ with
\[
|Df(x)|^{\,n}\le K\, J_f(x)\qquad \text{for a.e.\ }x\in D,
\]
where $Df(x)$ is the (a.e.\ defined) differential, $|Df(x)|$ is its operator norm, and
$J_f(x)=\det Df(x)$ (in particular $J_f>0$ a.e.).

A \emph{quasiconformal harmonic mapping} (often abbreviated \emph{HQC}) is a mapping that is
both harmonic and quasiconformal.
\\

The modern study of HQC mappings is often traced back to Martio's pioneering work
on harmonic quasiconformal self-maps of the disk \cite{Martio1969}.
A central theme is \emph{metric regularity}: under geometric or analytic hypotheses,
HQC homeomorphisms tend to be Lipschitz or even bi-Lipschitz (sometimes with respect to
Euclidean, hyperbolic, or quasihyperbolic-type metrics). In the planar case, sharp
boundary-to-interior principles and explicit distortion bounds can be obtained by combining
harmonic analysis tools (Poisson kernels, Hilbert transform, harmonic measure) with the
geometric theory of quasiconformal maps.

A prominent line of results concerns \emph{bi-Lipschitz characterizations} and boundary behavior.
For instance, Pavlovi\'{c} established an equivalence between quasiconformality of a harmonic
self-homeomorphism of $\D$ and Euclidean bi-Lipschitz behavior (with boundary criteria in terms
of the Hilbert transform) \cite{Pavlovic2002}. Partyka and Sakan developed explicit
bi-Lipschitz constants for harmonic quasiconformal self-maps of $\D$ \cite{PartykaSakan2007},
and later Partyka--Sakan--Zhu studied further structural constraints on HQC maps via conditions
on the analytic part \cite{PartykaSakanZhu2018}. Mateljevi\'{c} and Vuorinen investigated HQC
maps as quasi-isometries and Lipschitz mappings with respect to quasihyperbolic-type metrics
\cite{MateljevicVuorinen2010}. In higher dimensions, issues related to the possible vanishing
of the Jacobian complicate the picture; work of Bo\v{z}in and Mateljevi\'{c} illustrates this
via Jacobian bounds and co-Lipschitz phenomena for harmonic $K$-quasiconformal maps
\cite{BozinMateljevic2015}. Contributions of Arsenovi\'{c}, Koji\'{c}, Mateljevi\'{c}, Astala and
Todor\v{c}evi\'{c} form another important part of the literature, including
Lipschitz continuity of harmonic quasiregular maps and metric bi-Lipschitz properties
\cite{ArsenovicKojicMateljevic2008,Manojlovic2009,Todorcevic2019,AstalaManojlovic2015}.

Within this circle of ideas, a particularly influential body of work is due to
Kalaj and collaborators. One set of results treats \emph{boundary regularity and Lipschitz
continuity} for harmonic quasiconformal maps between planar domains with various smoothness
assumptions on the boundary (e.g.\ $C^{1,\alpha}$ or Dini smooth Jordan curves),
often yielding global Lipschitz (or bi-Lipschitz) estimates
\cite{Kalaj2008MZ,Kalaj2017Dini,Kalaj2022RMI}.
Another direction studies \emph{HQC solutions of elliptic equations} beyond $\Delta f=0$,
for example quasiconformal self-maps of $\D$ satisfying Poisson-type equations $\Delta f=g$
(and related Dirichlet data), establishing bi-Lipschitz bounds under analytic control of the
Laplacian \cite{KalajPavlovic2011,KalajSaksman2019,KalajZlaticanin2019}.
Kalaj also developed sharp inequalities for harmonic mappings and quantitative harmonic-analysis
estimates, such as Riesz-type inequalities \cite{Kalaj2019Riesz} and Heinz--Schwarz type bounds
in higher dimensions \cite{Kalaj2016HeinzSchwarz}, and explored links between HQC maps,
harmonic measure, and Muckenhoupt weights (yielding, among other things, an HQC version of
Lindel\"of-type boundary regularity) \cite{Kalaj2015Muckenhoupt}.

Historically, the broad theory of harmonic mappings (including univalence criteria via the
shear construction) provides essential background and motivation for many HQC questions;
a classical reference is \cite{ClunieSheilSmall1984}. For a modern perspective on the
interaction of quasiconformal mappings with elliptic PDE (including the Beltrami equation and
regularity theory), one may consult general sources such as \cite{AstalaIwaniecMartin2009},
which also contextualize why ``elliptic operator + boundary value problem'' viewpoints are so
effective in this area.

\medskip

\paragraph{\textbf{Organization of the paper.}}
In this paper we improve the result of \cite{GjokajKalaj2022} by showing that the same conclusion holds
under the sole assumption that the source domain has $C^{1,\alpha}$ boundary.
Section~\ref{sec:lipschitz} contains the analytic core of the argument.
We establish a basepoint boundary H\"older principle near a $C^{1,\alpha}$ boundary point $x_0$:
assuming a $C^{0,\mu}$ control of the trace on $\Gamma=\partial D\cap B(x_0,r_0)$,
\[
|u(\xi)-u(x_0)|\le M|\xi-x_0|^\mu,\qquad \xi\in\Gamma,
\]
we obtain quantitative interior control along the inward normal ray
$x=x_0+t\nu(x_0)$, $0<t<cr_0$. For $0<\mu<1$, Lemma~\ref{lem:anchored_holder_to_grad_normal} yields
\[
\sup_{y\in D\cap B(x,t/2)}|u(y)-u(x_0)|
\lesssim \Bigl(M+\frac{A}{r_0^\mu}\Bigr)t^\mu,
\qquad A:=\|u\|_{L^\infty(D\cap B(x_0,r_0))},
\]
and, in particular,
\[
|u(x)-u(x_0)|\lesssim \Bigl(M+\frac{A}{r_0^\mu}\Bigr)t^\mu,
\qquad
|\nabla u(x)|\lesssim \Bigl(M+\frac{A}{r_0^\mu}\Bigr)t^{\mu-1}.
\]
In the case $\mu>1$, the corresponding bounds are recorded in Lemma~\ref{lema2}:
\[
|u(x)-u(x_0)|\lesssim \Bigl(M\,r_0^{\mu-1}+\frac{A}{r_0}\Bigr)t,
\qquad
|\nabla u(x)|\lesssim \Bigl(M\,r_0^{\mu-1}+\frac{A}{r_0}\Bigr).
\]

Finally, combining these tools with a boundary exponent-improvement/iteration scheme, in Section~\ref{section4},
we obtain Lipschitz regularity up to the boundary for harmonic quasiconformal homeomorphisms
between bounded $C^{1,\alpha}$ domains (Theorem~\ref{thm:lipschitz_C1a}).
A brief remark on the planar case is included at the end of the section.

\medskip

\paragraph{\textbf{Further reading.}}
For background on harmonic quasiconformal mappings, boundary regularity, and related
Lipschitz/co-Lipschitz phenomena, we refer to
\cite{Martio1969,Pavlovic2002,PartykaSakan2007,PartykaSakanZhu2018,MateljevicVuorinen2010,
Manojlovic2009,ArsenovicKojicMateljevic2008,BozinMateljevic2015,Todorcevic2019,
ClunieSheilSmall1984,AstalaIwaniecMartin2009,
Kalaj2003Convex,Kalaj2004Convex,KalajPavlovic2005HalfPlane,KalajMateljevic2006Inner,
Kalaj2008MZ,Kalaj2008Ball,KalajPavlovic2011,Kalaj2011Distance,Kalaj2011Kneser,
KalajMateljevic2011PA,KalajVuorinen2012,Kalaj2013Gradient,KalajPonnusamyVuorinen2014,
Kalaj2015Muckenhoupt,Kalaj2016HeinzSchwarz,Kalaj2017Dini,KalajSaksman2019, Gjokaj, 
KalajZlaticanin2019,Kalaj2019Riesz,Kalaj2022RMI},
and the references therein.
\section{New results}
The authors  proved in \cite{GjokajKalaj2022} that if $f:\mathbf{B}^n\to\Omega\subset\mathbf{R}^n$
is a harmonic $K$--quasiconformal mapping and $\partial\Omega\in C^{1,\alpha}$,
then $f$ is Lipschitz continuous in $\mathbf{B}^n$.

In this work we improve that result by allowing an arbitrary $C^{1,\alpha}$ source domain $D$
(in place of the unit ball), while still assuming $C^{1,\alpha}$ regularity of the target boundary.

\begin{theorem}[Lipschitz regularity for harmonic quasiconformal maps]\label{thm:lipschitz_C1a}
Let $f:D\to\Omega$ be a $K$--quasiconformal harmonic mapping.
If $\partial D$ and $\partial\Omega$ are $C^{1,\alpha}$ with $\alpha\in(0,1)$, then $f$ is Lipschitz
on $\overline D$; equivalently, there exists $L>0$ such that
\[
|f(x)-f(y)|\le L\,|x-y|,\qquad \text{for all }x,y\in \overline D.
\]
\end{theorem}

\section{Auxiliary results}\label{sec:lipschitz}

\noindent As indicated above, we now prove a localized boundary-to-interior estimate showing that a basepoint
$\mu$-H\"older control of the boundary values at $x_0$ propagates into $\mu$-H\"older control in the interior and
yields the corresponding gradient bound $|\nabla u(x)|\lesssim t^{\mu-1}$ along the inward normal through $x_0$,
where $t$ denotes the normal distance from $x_0$.

\begin{lemma}[Basepoint boundary H\"older at $x_0$ $\Rightarrow$ interior H\"older and gradient control along the normal]
\label{lem:anchored_holder_to_grad_normal}
Let $D\subset\mathbf R^n$ be a domain and let $x_0\in\partial D$.
Assume that $\partial D$ is $C^{1,\alpha}$ in a neighborhood of $x_0$ for some $\alpha\in(0,1)$.
Fix $0<\mu< 1$.

Then there exist $r_0>0$, $c\in(0,1)$ and $C>0$ (depending only on the local $C^{1,\alpha}$ character
of $\partial D$ near $x_0$ and on $\mu$) such that the following holds.

Let $u$ be harmonic in $\Omega_{x_0}:=D\cap B(x_0,r_0)$ and continuous on $\overline{\Omega_{x_0}}$.
Assume that on the boundary patch $\Gamma:=\partial D\cap B(x_0,r_0)$,
\begin{equation}\label{eq:anchored_holder_bd_combined}
|u(\xi)-u(x_0)|\le M|\xi-x_0|^\mu,\qquad \xi\in\Gamma,
\end{equation}
and set $A:=\|u\|_{L^\infty(\Omega_{x_0})}$.
Let $\nu(x_0)$ be the inward unit normal to $\partial D$ at $x_0$.
\\
Then for every
\[
x=x_0+t\,\nu(x_0)\in D,\qquad 0<t<cr_0,
\]
we have the ``tent'' estimate
\begin{equation}\label{eq:tent_holder_combined}
\sup_{y\in D\cap B(x,t/2)} |u(y)-u(x_0)|
\;\le\; C\Big(M+\frac{A}{r_0^\mu}\Big)\,t^\mu.
\end{equation}
In particular,
\begin{equation}\label{eq:holder_normal_combined}
|u(x)-u(x_0)|
\;\le\; C\Big(M+\frac{A}{r_0^\mu}\Big)\,t^\mu,
\end{equation}
and the full gradient satisfies
\begin{equation}\label{eq:grad_normal_combined}
|\nabla u(x)|
\;\le\; C\Big(M+\frac{A}{r_0^\mu}\Big)\,t^{\mu-1}.
\end{equation}
Equivalently, writing $\delta_D(x):=\dist(x,\partial D)$, and noting that
$\delta_D(x)=t$ along this inward normal ray for $t<cr_0$, we obtain
\[
|u(x)-u(x_0)|\le C\Big(M+\frac{A}{r_0^\mu}\Big)\,\delta_D(x)^\mu,
\qquad
|\nabla u(x)|\le C\Big(M+\frac{A}{r_0^\mu}\Big)\,\delta_D(x)^{\mu-1}.
\]
\end{lemma}

\begin{proof}
Translate/rotate so that $x_0=0$ and $\nu(x_0)=e_n$. Rescale by $x=r_0 z$; it suffices to prove the
claim for $r_0=1$ and then scale back.
Thus set
\[
\Omega_0:=D\cap B(0,1),\qquad \Gamma:=\partial D\cap B(0,1),\qquad
\Sigma:=\partial\Omega_0\setminus\Gamma=\partial B(0,1)\cap \overline{D},
\]
and $A:=\|u\|_{L^\infty(\Omega_0)}$.
Choose $c\in(0,1)$ so small that $B(0,2c)\cap\Sigma=\varnothing$ and that the segment
$\{te_n:0<t<2c\}\subset D$ has $0$ as unique nearest boundary point (possible by $C^1$ regularity).

Fix $t\in(0,c)$ and set $x:=te_n$. Then $\delta_D(x)=t$.
Moreover, for every $y\in B(x,t/2)\cap D$ one has $|y|\lesssim t$ and $\delta_D(y)\simeq t$
(with constants depending only on the local $C^1$ geometry).

\begin{figure}[h]
    \centering
    \includegraphics[scale=0.15]{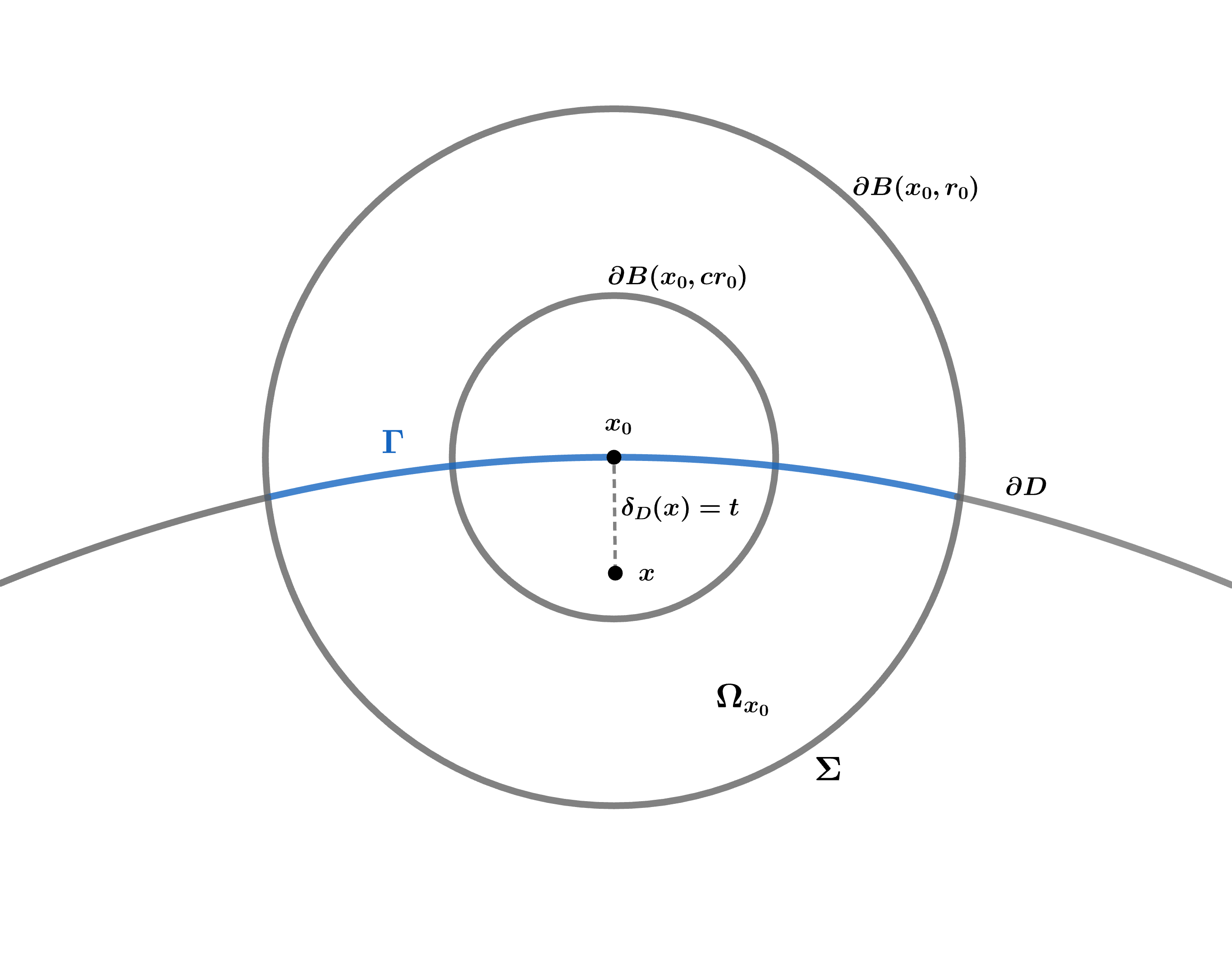}
\caption{\textbf{Local geometry near a $C^{1,\alpha}$ boundary patch.}
The boundary of $\Omega_{x_0}=D\cap B(x_0,r_0)$ is decomposed as $\partial\Omega_{x_0}=\Gamma\cup\Sigma$,
where $\Gamma=\partial D\cap B(x_0,r_0)$ and $\Sigma=\partial B(x_0,r_0)\cap \overline{D}.$}
\label{22}
\end{figure}\ \\[0.2cm]
Let $\omega^y$ be harmonic measure for $\Omega_0$ with pole at $y$. Since $u$ is harmonic in $\Omega_0$
and continuous on $\partial\Omega_0$,
\[
u(y)-u(0)=\int_{\partial\Omega_0}(u(\zeta)-u(0))\,d\omega^y(\zeta).
\]
Using \eqref{eq:anchored_holder_bd_combined} on $\Gamma$ and $|u(\zeta)-u(0)|\le 2A$ on $\Sigma$,
\begin{equation}\label{eq:basic_y_combined}
|u(y)-u(0)|
\le
M\int_{\Gamma}|\zeta|^\mu\,d\omega^y(\zeta)
+2A\,\omega^y(\Sigma).
\end{equation}

\emph{$\Sigma$ contribution.}
Let $w$ be harmonic in $\Omega_0$ with $w=0$ on $\Gamma$ and $w=1$ on $\Sigma$.
Then $w(y)=\omega^y(\Sigma)$.
Since $B(0,2c)\cap\Sigma=\varnothing$, we have $w\equiv0$ on $\partial\Omega_0\cap B(0,2c)$.
The boundary Lipschitz estimate in $C^{1,\alpha}$ domains (after flattening/scaling, see
\cite[Corollary~8.36]{GT}) gives
\[
\|\nabla w\|_{L^\infty(\Omega_0\cap B(0,c))}\le C,
\]
hence $w(y)\le C\,\delta_D(y)\lesssim t$, i.e.
\begin{equation}\label{eq:leak_y_combined}
\omega^y(\Sigma)\le Ct.
\end{equation}

\emph{$\Gamma$ contribution.}
For $s>0$ define $F_y(s):=\omega^y(\Gamma\setminus \overline{B}(0,s))$.
Since $\Gamma\subset B(0,1)$, $F_y(s)=0$ for $s\ge1$.
By layer-cake,
\begin{equation}\label{eq:layercake_y_combined}
\int_{\Gamma}|\zeta|^\mu\,d\omega^y(\zeta)=\mu\int_0^1 s^{\mu-1}F_y(s)\,ds.
\end{equation}
For $s\in(\delta_D(y),1]$ one has the decay
\begin{equation}\label{eq:Fy_decay_combined}
F_y(s)\le C\,\frac{\delta_D(y)}{s},
\end{equation}
proved by comparing $F_y(s)$ with the harmonic function that vanishes on $\Gamma\cap B(0,s)$ and
equals $1$ on $\partial\Omega_0\setminus(\Gamma\cap B(0,s))$, then applying a boundary Lipschitz estimate
after flattening/scaling (again of the type \cite[Corollary~8.36]{GT}).
Splitting \eqref{eq:layercake_y_combined} at $s=\delta_D(y)$ and using $0<\mu<1$ yields
\[
\int_{\Gamma}|\zeta|^\mu\,d\omega^y(\zeta)
\le \mu\int_0^{\delta_D(y)} s^{\mu-1}\,ds
+ C\mu\,\delta_D(y)\int_{\delta_D(y)}^1 s^{\mu-2}\,ds
\le C\,\delta_D(y)^\mu
\lesssim t^\mu.
\]
Together with \eqref{eq:leak_y_combined} and \eqref{eq:basic_y_combined}, this gives
\[
|u(y)-u(0)|\le C(Mt^\mu+At)\le C(M+A)\,t^\mu,
\]
since $0<t<1$ and $0<\mu<1$ imply $t\le t^\mu$. Taking the supremum over $y\in B(x,t/2)\cap D$
yields \eqref{eq:tent_holder_combined} (for $r_0=1$). In particular, taking $y=x$ gives
\eqref{eq:holder_normal_combined} (for $r_0=1$).
\\ \
Finally, apply the interior gradient estimate at $x$:
\[
|\nabla u(x)|\le \frac{C}{t}\,\sup_{B(x,t/2)\cap D}|u-u(0)|
\le C(M+A)\,t^{\mu-1},
\]
which is \eqref{eq:grad_normal_combined} for $r_0=1$.
\\ \
Rescaling back to general $r_0$ gives the factor $M+\frac{A}{r_0^\mu}$ in
\eqref{eq:tent_holder_combined}--\eqref{eq:grad_normal_combined}, and completes the proof.
\end{proof}

\medskip
\noindent \\
We do not treat the endpoint case $\mu=1$ in Lemma~\ref{lem:anchored_holder_to_grad_normal}, since the
harmonic-measure estimate of $\int_{\Gamma}|\zeta|^\mu\,d\omega^y(\zeta)$ then produces the logarithmic factor
\[
\int_{\delta_D(y)}^{1}\frac{ds}{s}=\log\!\frac{1}{\delta_D(y)},
\]
and hence only yields logarithmic control of the gradient. We therefore avoid the endpoint $\mu=1$ in the
iteration. Once the exponent satisfies $\mu>1$, the logarithmic term disappears and the same computation yields
a uniform gradient bound. For convenience, we record the corresponding estimate in the following lemma.

\begin{lemma}[Addendum to Lemma~\ref{lem:anchored_holder_to_grad_normal}: the case $\mu>1$]
\label{lema2}
Assume the hypotheses and notation of Lemma~\ref{lem:anchored_holder_to_grad_normal}, except that
\eqref{eq:anchored_holder_bd_combined} holds for some exponent $\mu>1$.
Then there exist $r_0>0$, $c\in(0,1)$ and $C>0$ (depending only on the local $C^{1,\alpha}$ character
of $\partial D$ near $x_0$ and on $\mu$) such that for every
\[
x=x_0+t\,\nu(x_0)\in D,\qquad 0<t<cr_0,
\]
we have
\begin{equation}\label{eq:holder_mu_gt_1}
|u(x)-u(x_0)|
\;\le\; C\Big(M\,r_0^{\mu-1}+\frac{A}{r_0}\Big)\,t,
\end{equation}
and
\begin{equation}\label{eq:grad_mu_gt_1}
|\nabla u(x)|
\;\le\; C\Big(M\,r_0^{\mu-1}+\frac{A}{r_0}\Big).
\end{equation}
Equivalently, along this inward normal ray (where $\delta_D(x)=t$ for $t<cr_0$),
\[
|u(x)-u(x_0)|
\le C\Big(M\,r_0^{\mu-1}+\frac{A}{r_0}\Big)\,\delta_D(x),
\qquad
|\nabla u(x)|
\le C\Big(M\,r_0^{\mu-1}+\frac{A}{r_0}\Big).
\]
\end{lemma}
\begin{proof}[Proof]
The proof is identical to that of Lemma~\ref{lem:anchored_holder_to_grad_normal}, except for the estimate of
\[
\int_{\Gamma}|\zeta|^\mu\,d\omega^y(\zeta).
\]
Using the same notation and the layer--cake representation,
\[
\int_{\Gamma}|\zeta|^\mu\,d\omega^y(\zeta)=\mu\int_0^1 s^{\mu-1}F_y(s)\,ds,
\]
and the same decay estimate \eqref{eq:Fy_decay_combined} for $s\in(\delta_D(y),1]$,
we split at $s=\delta_D(y)$ to obtain
\begin{align*}
\int_{\Gamma}|\zeta|^\mu\,d\omega^y(\zeta)
&\le \mu\int_0^{\delta_D(y)} s^{\mu-1}\,ds
   + C\mu\,\delta_D(y)\int_{\delta_D(y)}^{1} s^{\mu-2}\,ds \\
&= \delta_D(y)^\mu
   + C\mu\,\delta_D(y)\,\frac{1-\delta_D(y)^{\mu-1}}{\mu-1}
\;\le\; C\,\delta_D(y),
\end{align*}
since $\mu>1$ and $\delta_D(y)\in(0,1)$. Substituting this bound in place of the corresponding estimate
in Lemma~\ref{lem:anchored_holder_to_grad_normal} yields the stated conclusions.
\end{proof}

\section{The proof of the main result}\label{section4}

\begin{proof}[Proof of Theorem~\ref{thm:lipschitz_C1a}]
We prove a uniform bound for $|Df|$ in a collar neighborhood of $\partial D$ and then combine it
with standard interior estimates.

\medskip
\noindent\textbf{0) Local charts and tubular neighborhoods.}
Fix $x_0\in\partial D$ and set $q_0:=f(x_0)\in\partial\Omega$,
using the continuous extension of $f$ to $\overline D$ provided in Step~1.

Since $\partial D$ is $C^{1,\alpha}$ near $x_0$, there exists a neighborhood $U\subset\R^n$ of $x_0$
and $r_D>0$ such that every point
\[
x\in \mathcal C_D:=\{x\in D\cap U:\delta_D(x)<r_D\}
\]
has a unique nearest boundary point $\pi(x)\in\partial D\cap U$ and admits the normal
representation
\begin{equation}\label{eq:tubular_D}
x=\pi(x)+\delta_D(x)\,\nu_D(\pi(x)),
\end{equation}
where $\nu_D(\xi)$ denotes the inward unit normal to $\partial D$ at $\xi$.

Similarly, since $\partial\Omega$ is $C^{1,\alpha}$ and compact, we may choose a neighborhood $V$
of $\partial\Omega$ and constants $r_\Omega>0$, $C_\Omega>0$ with the following property:
for every $q\in\partial\Omega$ there is a rigid motion $T_q(z)=R_q(z-q)$ (with $R_q$ orthogonal)
such that in the coordinates $y=T_q(z)$ the boundary $\partial\Omega$ near $q$ is a $C^{1,\alpha}$
graph
\[
\partial\Omega_q:=T_q(\partial\Omega)\cap B(0,r_\Omega)=\{(y',y_n): y_n=\Phi_q(y')\},
\]
where $\Phi_q\in C^{1,\alpha}$, $\Phi_q(0)=0$, $\nabla\Phi_q(0)=0$, and uniformly for all such $q$,
\begin{equation}\label{eq:Phi_uniform}
|\nabla\Phi_q(y')|\le C_\Omega |y'|^\alpha,
\qquad
|\Phi_q(y')|\le C_\Omega |y'|^{1+\alpha},
\qquad y'\in B(0,r_\Omega).
\end{equation}
(These follow from $\nabla\Phi_q(0)=0$ and uniform $C^{1,\alpha}$ control of $\partial\Omega$.)

\medskip
\noindent\textbf{1) Global H\"older control via Gehring--Martio.}
Since $D$ and $\Omega$ are bounded $C^{1,\alpha}$ domains,
they are bounded Lipschitz domains. It is a standard fact that every bounded Lipschitz domain is a (bounded)
uniform domain (see, e.g., \cite{BuckleyHerron,Vaisala} and the references therein). Thus both $D$ and $\Omega$ are uniform domains. Moreover, by Lemma~2.18 in \cite{GehringMartio}, every bounded uniform domain is a John domain, so in particular
$\Omega$ is a John domain. 

Therefore the hypotheses of Theorem~3.28 in \cite{GehringMartio} apply: since $f$ maps a uniform domain $D$
onto a John domain $\Omega$, there exists an exponent $\beta\in(0,1)$ and a constant $C>0$ such that
\begin{equation}\label{betaholder}
|f(x)-f(y)|\le C\,|x-y|^{\beta},\qquad x,y\in \overline{D}.
\end{equation}
Thus $f$ is globally $\beta$--H\"older continuous on $\overline{D}$.
Set $\beta_0:=\beta$. In particular, $f$ admits a continuous extension to $\overline D$ (still denoted by $f$).

\medskip
\noindent\textbf{2) One-step exponent improvement for a rotated normal component (at each boundary point).}
Fix $\xi\in\partial D\cap U$ and set $q:=f(\xi)\in\partial\Omega$.
Let $T_q$ be the rigid motion from Step~0 and define the rotated map
\[
f^{(q)}:=T_q\circ f,\qquad F^{(q)}:=f^{(q)}|_{\partial D\cap U}.
\]
Then $f^{(q)}$ is still harmonic and $K$--quasiconformal, and $F^{(q)}(\partial D\cap U)$ lies in
$\partial\Omega_q$ (after shrinking $U$ so that $f(D\cap U)\subset V$).
Writing $f^{(q)}=(\tilde f^{(q)}, f^{(q)}_n)$, the boundary identity is
\begin{equation}\label{eq:boundary_identity_q}
F^{(q)}_n(\eta)=\Phi_q(\tilde F^{(q)}(\eta)),\qquad \eta\in\partial D\cap U.
\end{equation}
Since $F^{(q)}(\xi)=0$, we have $\tilde F^{(q)}(\xi)=0$ and $F^{(q)}_n(\xi)=0$.
Combining \eqref{eq:boundary_identity_q} with \eqref{eq:Phi_uniform} gives, for $\eta\in\partial D\cap U$,
\[
|F^{(q)}_n(\eta)-F^{(q)}_n(\xi)|
=|\Phi_q(\tilde F^{(q)}(\eta))|
\le C_\Omega |\tilde F^{(q)}(\eta)|^{1+\alpha}.
\]
Because $T_q$ is an isometry, \eqref{betaholder} implies
$|\tilde F^{(q)}(\eta)|\le |F^{(q)}(\eta)|=|f(\eta)-f(\xi)|\le C_0|\eta-\xi|^{\beta_0}$.
Hence
\begin{equation}\label{eq:improved_normal_trace}
|F^{(q)}_n(\eta)-F^{(q)}_n(\xi)|
\le C_1 |\eta-\xi|^{\beta_1},
\qquad
\beta_1:=(1+\alpha)\beta_0,
\end{equation}
with $C_1:=C_\Omega C_0^{1+\alpha}$ (uniform in $\xi$).

\medskip
\noindent\textbf{3) Boundary gradient estimate for $f^{(q)}_n$ and passage to $|Df|$.}
Recall from Step~2 that \eqref{eq:improved_normal_trace} holds with exponent $\beta_1=(1+\alpha)\beta_0$.
Assume for the moment that $\beta_1<1$.
Apply Lemma~\ref{lem:anchored_holder_to_grad_normal} to the harmonic function $u=f^{(q)}_n$ at $x_0=\xi$
with exponent $\mu=\beta_1$.
\\
Then there exist
$r_1\in(0,r_D)$, $c\in(0,1)$, and $C_2>0$ (uniform in $\xi$) such that for
\[
x=\xi+t\,\nu_D(\xi)\in D,\qquad 0<t<c r_1,
\]
we have the gradient estimate
\begin{equation}\label{eq:grad_fnq}
|\nabla f^{(q)}_n(x)|\le C_2\, t^{\beta_1-1}.
\end{equation}
We now reduce this to a bound for $|Df|$. Let $A:=Df^{(q)}(x)$.
Since $A$ and $A^T$ have the same singular values, we have
\[
l(A)=\inf_{|b|=1}|A^Tb|\le |A^T e_n|=|\nabla f^{(q)}_n(x)|,
\]
where $l(A)$ denotes the smallest singular value and $e_n$ is the $n$-th coordinate vector. By quasiconformality,
\begin{equation}\label{eq:qc_distortion_again}
|Df^{(q)}(x)|\le H(K)\,l(Df^{(q)}(x)),
\end{equation}
hence
\begin{equation}\label{eq:Df_by_gradfnq_simple}
|Df^{(q)}(x)|\le H(K)\,|\nabla f^{(q)}_n(x)|.
\end{equation}

Since $T_q$ is a rigid motion, $Df^{(q)}(x)=DT_q(f(x))\,Df(x)$ and hence
$|Df^{(q)}(x)|=|Df(x)|$. Combining \eqref{eq:Df_by_gradfnq_simple} with
\eqref{eq:grad_fnq} we obtain
\[
|Df(x)|\le C\,t^{\beta_1-1}
\]
for all $x\in D$ of the form $x=\xi+t\nu_D(\xi)$ with $\xi\in\partial D\cap U$
and $0<t<c r_1$, uniformly in $\xi$. By the tubular representation
\eqref{eq:tubular_D}, for such points one has $\delta_D(x)=t$, and therefore
\begin{equation}\label{eq:Df_collar_in_terms_delta_new}
|Df(x)|\le C\,\delta_D(x)^{\beta_1-1},
\qquad x\in D,\ \delta_D(x)<c r_1.
\end{equation}

\medskip
\noindent\textbf{4) Improving the H\"older exponent of the full boundary trace.}
Let $\xi,\eta\in\partial D\cap U$ and set $\rho:=|\xi-\eta|$.
For $\rho$ small, $C^{1,\alpha}$ regularity of $\partial D$ implies there exists a rectifiable curve
$\sigma\subset \partial D\cap U$ joining $\xi$ to $\eta$ with $\mathrm{length}(\sigma)\le C_D\rho$.

Consider the three-piece path in $D$:
\[
\gamma_1(s):=\xi+s\,\nu_D(\xi),\quad s\in[0,\rho],
\]
\[
\gamma_2:=\{\ \sigma(\tau)+\rho\,\nu_D(\sigma(\tau)):\ \tau\ \text{parametrizes }\sigma\ \},
\]
\[
\gamma_3(s):=\eta+(\rho-s)\,\nu_D(\eta),\quad s\in[0,\rho].
\]
Then $\gamma_1$ joins $\xi$ to $\xi+\rho\nu_D(\xi)$, $\gamma_2$ joins $\xi+\rho\nu_D(\xi)$ to
$\eta+\rho\nu_D(\eta)$, and $\gamma_3$ joins $\eta+\rho\nu_D(\eta)$ to $\eta$; hence
$\gamma:=\gamma_1\cup\gamma_2\cup\gamma_3$ connects $\xi$ to $\eta$ inside $D$.

Moreover, for $\rho$ small (depending only on the local tubular radius), we have
\[
\delta_D(\gamma_1(s))=s,\qquad \delta_D(\gamma_3(s))=s,
\]
and along $\gamma_2$ the distance to the boundary is constant:
\begin{equation}\label{eq:cigar_distance_new}
\delta_D(z)= \rho \qquad\text{for all }z\in\gamma_2,
\end{equation}
and $\mathrm{length}(\gamma_2)\le C\rho$.
\\
Integrating \eqref{eq:Df_collar_in_terms_delta_new} along $\gamma$ gives
\[
|f(\xi)-f(\eta)|
\le \int_\gamma |Df|\,ds
\le \int_{\gamma_1} C\,\delta_D^{\beta_1-1}\,ds
     +\int_{\gamma_2} C\,\delta_D^{\beta_1-1}\,ds
     +\int_{\gamma_3} C\,\delta_D^{{\beta_1}-1}\,ds.
\]
Using $\delta_D(\gamma_1(s))=s$ and $\delta_D(\gamma_3(s))=s$,
\[
\int_{\gamma_1} C\,\delta_D^{{\beta_1}-1}\,ds
= C\int_0^\rho s^{{\beta_1}-1}\,ds
\le C\rho^{\beta_1},
\qquad
\int_{\gamma_3} C\,\delta_D^{{\beta_1}-1}\,ds \le C\rho^{\beta_1}.
\]
Using \eqref{eq:cigar_distance_new} and $\mathrm{length}(\gamma_2)\le C\rho$,
\[
\int_{\gamma_2} C\,\delta_D^{{\beta_1}-1}\,ds
\le C\,\rho^{{\beta_1}-1}\,\mathrm{length}(\gamma_2)
\le C\rho^{\beta_1}.
\]
Therefore
\[
|f(\xi)-f(\eta)|\le C\,\rho^{\beta_1}=C|\xi-\eta|^{\beta_1},
\]
so on $\partial D\cap U$ we have improved the boundary exponent of $F=f|_{\partial D}$ from
$\beta_0$ to $\beta_1$.

\medskip
\noindent\textbf{5) Iteration and final step.}
Define inductively $\beta_{m+1}:=(1+\alpha)\beta_m$ for $m\ge 0$, with $\beta_0=\beta$.
If necessary, decrease $\beta$ slightly so that $\beta_m\neq 1$ for all $m$; this is possible since the
global H\"older estimate \eqref{betaholder} remains valid for any smaller exponent.
\\
Assume $\beta_m<1$ and that on $\partial D\cap U$,
\begin{equation}\label{eq:trace_holder_stage_m}
|f(\xi)-f(\eta)|\le C_m|\xi-\eta|^{\beta_m}.
\end{equation}
Repeating Steps~2--4 with $\beta_0$ replaced by $\beta_m$ yields
\[
|f(\xi)-f(\eta)|\le C_{m+1}|\xi-\eta|^{\beta_{m+1}},
\qquad \beta_{m+1}=(1+\alpha)\beta_m,
\]
and
\[
|Df(x)|\le \widetilde C_{m+1}\,\delta_D(x)^{\beta_{m+1}-1}
\]
for $x$ in a (possibly smaller) collar neighborhood of $\partial D\cap U$.
\\

Choose $m_0\in\mathbf{N}$ minimal such that $\beta_{m_0}>1$.
At this stage we apply Lemma~\ref{lema2} in Step~3 and obtain a uniform bound
\begin{equation}\label{eq:Df_bounded_collar}
|Df(x)|\le C_*\qquad\text{for }x\text{ with }\delta_D(x)<\delta_*,
\end{equation}
in some collar $\{x\in D:\delta_D(x)<\delta_*\}$ near $\partial D\cap U$ (with $\delta_*>0$).

\medskip
\noindent\textbf{6) Globalization and Lipschitz continuity.}
Cover $\partial D$ by finitely many boundary patches $U^{(1)},\dots,U^{(N)}$ and apply the above
argument on each patch. This yields a uniform bound \eqref{eq:Df_bounded_collar} on the global collar
\[
\mathcal C:=\{x\in D:\delta_D(x)<\delta_*\}.
\]
Set
\[
D_{\delta_*}:=\{x\in \overline D:\delta_D(x)\ge \delta_*\}.
\]
Since $\delta_D$ is continuous and $\overline D$ is compact, $D_{\delta_*}$ is compact. Moreover,
$\delta_*>0$ implies $D_{\delta_*}\subset D$. Because $f$ is harmonic in $D$, it is smooth, hence $Df$ is
continuous on $D_{\delta_*}$, and therefore
\[
\sup_{D_{\delta_*}}|Df|=\max_{D_{\delta_*}}|Df|<\infty.
\]
Combining this with \eqref{eq:Df_bounded_collar} gives $\sup_D|Df|<\infty$.

Finally, boundedness of $|Df|$ implies that $f$ is Lipschitz on $\overline D$. Indeed, since $D$ is a bounded $C^{1}$ (hence Lipschitz) domain, it is uniform and therefore quasiconvex; in particular,
any $x,y\in D$ can be joined by a rectifiable curve $\Gamma\subset D$ with $\mathrm{length}(\Gamma)\le C_D\,|x-y|$. Hence
\[
|f(x)-f(y)|\le \int_\Gamma |Df|\,ds \le (\sup_D|Df|)\,C_D\,|x-y|.
\]
Passing to $x,y\in\overline D$ by continuity completes the proof.
\end{proof}

\begin{remark}
In the planar case, Lipschitz continuity of harmonic quasiconformal mappings
between Jordan domains with $C^{1,\alpha}$ boundaries was proved in \cite{Kalaj2008MZ}
by reducing the problem to the unit disk and using Kellogg's theorem for conformal mappings
together with the Poisson representation formula.
\end{remark}
\section*{Acknowledgements}
The authors gratefully acknowledge financial support from the grants "Mathematical Analysis, Optimisation and Machine Learning" and "Complex-analytic and geometric techniques for non-Euclidean machine learning: theory and applications".

\end{document}